\documentclass[letterpaper, 10 pt, conference]{ieeeconf}
\IEEEoverridecommandlockouts  
\usepackage{graphics,epstopdf,wrapfig}
\usepackage{graphicx}
\usepackage{xcolor}
\usepackage{epsfig}
\usepackage{subfig}
\usepackage{amsmath}
\usepackage{mathrsfs}
\usepackage{epsfig}
\usepackage{multirow}
\usepackage{amssymb,latexsym,amsfonts,amsmath}
\usepackage{graphicx}
\usepackage{color}
\usepackage{epstopdf}
\usepackage{float}
\usepackage{tikz}
\usepackage{caption}
\usepackage{verbatim}
\usepackage{tikz-3dplot}
\usetikzlibrary{shapes.geometric}
\usepackage{environ}
\NewEnviron{myequation}{%
\begin{equation}
\scalebox{1}{$\BODY$}
\end{equation}
}

\tdplotsetmaincoords{70}{30}
\usepackage{algpseudocode}
\usepackage{mathtools}
\usepackage[lined,ruled,commentsnumbered]{algorithm2e}

\makeatletter
\newcommand{\removelatexerror}{\let\@latex@error\@gobble}
\makeatother
\usepackage{amsthm}
\usepackage{pgfplots}
\pgfdeclarelayer{background}
\pgfdeclarelayer{foreground}
\pgfsetlayers{background,main,foreground}
\usetikzlibrary{intersections}
\usetikzlibrary{patterns}
\usetikzlibrary{automata,positioning,shapes,arrows}

\newtheorem{thm}{Theorem}

\newtheorem{corollary}{Corollary}
 
\usepackage{tikz}
\usetikzlibrary{automata,positioning,shapes,arrows}
\usetikzlibrary{shapes,snakes}
\tikzstyle{block} = [draw, rectangle, minimum size=3em]
\tikzstyle{bigblock} = [draw, rectangle, minimum height=7em, minimum width=11em]

\theoremstyle{remark}

\tikzset{>=latex}
\newenvironment{definition}[1][Definition]{\begin{trivlist}
		\item[\hskip \labelsep {\bfseries #1}]}{\end{trivlist}}
\newenvironment{problem}[1][Problem]{\begin{trivlist}
		\item[\hskip \labelsep {\bfseries #1}]}{\end{trivlist}}

\theoremstyle{exampstyle}

\title{\LARGE\bf Switched Linear Systems Meet Markov Decision Processes:\\Stability Guaranteed Policy Synthesis}

\author{Bo~Wu, Murat Cubuktepe, and Ufuk Topcu
	\thanks{ Bo Wu, Murat Cubuktepe, and Ufuk Topcu are with the Department of Aerospace Engineering
and Engineering Mechanics, and the Oden Institute for Computational
Engineering and Sciences, University of Texas, Austin, 201 E 24th
St, Austin, TX 78712. email: {\tt\small $\{$bwu3, mcubuktepe, utopcu$\}$@utexas.edu}}}

\begin{document}
\maketitle

\begin{abstract}
Switched linear systems are time-varying nonlinear systems whose dynamics switch between different modes, where each mode corresponds to different linear dynamics.  They arise naturally to model unexpected failures, environment uncertainties or system noises during system operation. In this paper, we consider a special class of switched linear systems where the mode switches are governed by Markov decision processes (MDPs). We study the problem of synthesizing a policy in an MDP that stabilizes the switched system. Given a policy, the switched linear system becomes a Markov jump linear system whose stability conditions have been extensively studied. We make use of these stability conditions and propose three different computation approaches to find the stabilizing policy in an MDP. We derive our first approach by extending the stability conditions for a Markov jump linear system to handle switches governed by an MDP. This approach requires finding a feasible solution of a set of bilinear matrix inequalities, which makes policy synthesis typically challenging. To improve scalability, we provide two approaches based on convex optimization. We give three examples to show and compare our proposed solutions. 
\end{abstract}

\section{Introduction}
In recent years, we have witnessed an increasing research interest in switched linear systems \cite{liberzon2003switching,sun2006switched}, which consist of a set of subsystems (also known as modes) with linear dynamics and a switching logic that describes all the possible switches between modes. 
Such systems model engineering systems with multi-controllers, abrupt system parameter variations due to environmental uncertainties and sudden change in system structure because of system failures \cite{sun2005analysis}. 
Switched linear systems find its application in robotics \cite{zegers2018distributed}, wireless sensor networks \cite{zhang2017energy}, networked control systems \cite{cetinkaya2018analysis}, security and privacy \cite{wu2018privacy}.

Generally speaking, in switched systems there are two kinds of switching logic, namely autonomous and controlled ones  \cite{zhang2008l_}. The former could be the result of system's own characteristics or the environment and the latter is caused by designer's objectives. In this paper, we introduce a new system modeling framework for switched linear systems, where the switching logic is governed by both autonomous and controlled factors characterized by a Markov decision process (MDP) \cite{puterman2014markov}. 

The switching logic characterized by an MDP consists of a set of modes, a set of actions controlled by the designer, and a  transition relation that defines the probability of transiting from the current mode to next mode when taking a particular action. 
Therefore, such a switching logic captures both the designer's control by the action selection and the environment uncertainties that result in probabilistic mode switches. 
For example, in a multi-agent system, each agent needs coordination to achieve some global agreement with locally available information subject to possible communication link failure and creation \cite{olfati2004consensus}. The designer may decide to switch among a finite set of possible formations in terms of relative distances between agents. The task here is to design a switching logic such that the agents are able to coordinate with each other to achieve a certain task.  However, each formation change may result in different network topology probabilistically due to uncertainties in wireless communication. Therefore, the system has to make its formation switch decisions wisely to remain stable.


The major objective of this paper is to synthesize a policy in an MDP such that the switched system is stable. Given a policy, an MDP reduces to a discrete time Markov chain (DTMC) and therefore, the switched linear system becomes a Markov jump linear system (MJLS) \cite{costa2006discrete}, where the modes in the system switch probabilistically following a DTMC. The stability conditions for an MJLS have been extensively studied, see e.g. \cite{costa2006discrete,bolzern2015positive,shi2015survey,saravanakumar2017stability}.

We first show that only considering deterministic and stationary policies in an MDP, which can achieve maximum expected reward \cite{puterman2014markov}, is not sufficient to stabilize the system. Based on different stability conditions for an MJLS, we introduce three different approaches to finding the stabilizing policy. The first approach extends the stability conditions for an MJLS whose switches are governed by an induced DTMC after applying a policy to an MDP. The approach relies on finding a policy and a Lyapunov function simultaneously that gives a certificate of the stability of a MJLS. It involves solving for a set of bilinear matrix inequalities, which are intractable to solve in general~\cite{vanantwerp2000tutorial}. We also provide a sufficient condition for computation of a policy that stabilizes the system based on semidefinite programming and coordinate descent, which can be solved more efficiently in polynomial time using interior point methods~\cite{nesterov1994interior}. The sufficient condition based on semidefinite programming involves searching for a diagonal Lyapunov function that guarantees the stability. As it is only a sufficient condition, we propose another approach based on coordinate descent~\cite{razaviyayn2013unified,shen2017disciplined}. In each step, we update the variables with in the coordinate descent method to improve the convergence. Our experiments show that coordinate descent method outperforms the semidefinite relaxation and directly solving for the bilinear matrix inequalities. 

The rest of this paper is organized as follows. Section \ref{sec:preliminaries} introduces our modeling framework and necessary definitions. Section \ref{sec:Problem Formulation} formulates our policy synthesis problem. Three different solutions are proposed in Section \ref{sec:Stability Guaranteed Policy Synthesis} with respect to different stability conditions. Section \ref{sec:Examples} provides three examples to show the validity of our proposed solutions and compare their performances. Section \ref{sec:Conclusion} concludes the paper and discusses future directions.


\emph{Notation:} $|S|$ denotes the cardinality of a set $S$. Given a real matrix $A\in\mathbb{R}^{m\times n}$, $A'$ denotes its transpose. If $m=n$, $\rho(A)$ represents the spectral radius of $A$, i.e., $\rho(A)=\max_i|\lambda_i|$ where $\lambda_i,i\in\{1,\ldots,n\}$ are eigenvalues of $A$. $A>0$ ($A\geq 0$) denotes that the matrix $A$ is positive definite (positive semidefinite). $E[.]$ stands for computing the expectation. $\otimes$ denotes the Kronecker product. 
For $A_i\in\mathbb{R}^{n\times n},i\in\{1,\ldots,n\}$, we set $diag(A_i)\in\mathbb{R}^{Nn\times Nn}$ the block diagonal matrix formed with $A_i$ at the diagonal and zero anywhere else, i.e.
$$
diag(A_i)=\begin{bmatrix}
    A_1 & 0 & 0  \\
    0 & \ddots & 0\\
    0 & 0 & A_N
\end{bmatrix}.
$$

\section{Preliminaries}\label{sec:preliminaries}
In this section, we describe preliminary notions and definitions used in the sequel.


\subsection{Switched Linear Systems}
A switched linear system  \cite{liberzon2003switching} switches between different modes $s\in S=\{1,\ldots,|S|\}$, where there is a different linear dynamic corresponds to each mode $s$.  Mathematically, a discrete time switched linear system is described by 
\begin{equation}\label{eq:switched systems}
    x(k+1)=A_{s_k}x(k)+B_{s_k}w(k),
\end{equation}
where $x(k)\in\mathbb{R}^n$ is the state vector, $A_{s_k}\in\mathbb{R}^{n\times n}$ and $B_{s_k}\in\mathbb{R}^{n\times m}$ implies a matrix $A\in \{A_1,\ldots,A_{|S|}\}$ and matrix  $B\in \{B_1,\ldots,B_{|S|}\}$, respectively.  The linear dynamics of (\ref{eq:switched systems}) is given by matrices $A_i$ and $B_i$ when $s_k=i$, i.e, the  mode that the system is in  at time $k$. $w(k)$  denotes an i.i.d random noise with  mean $\mu_w$ and variance $R_w$.

The system in (\ref{eq:switched systems}) in its general form is a hybrid system where the mode switches could depend on both the continuous dynamics and discrete mode. Such a hybrid system has been extensively studied \cite{liberzon2003switching,lin2009stability}.  In this paper, we consider the cases where the mode switches are governed by a Markov decision process whose transitions only depends on discrete modes but are independent from the continuous state values.


\subsection{Markov Decision Processes }
 Formally, a Markov decision process (MDP) \cite{puterman2014markov} is defined as follows.
\begin{definition}
	An MDP is a tuple $\mathcal{M}=(S,\hat{s},\Sigma,T)$ which includes a finite set $S$ of states,
 an initial state $\hat{s}$,  a finite set $\Sigma$ of actions. 
		$T:S\times \Sigma\times S\rightarrow [0,1]$ is the probabilistic transition function with $
		T(s,\sigma,s'):=p(s'|s,\sigma),   \text{for}\; s,s'\in S, \sigma\in \Sigma$. We denote the number of modes, i.e., $|S|$ as $N$. 
\end{definition}

	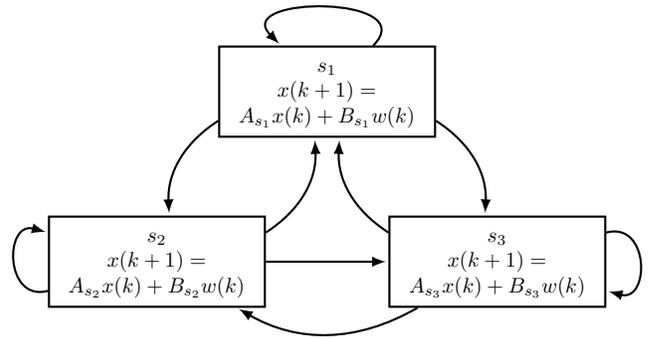
\begin{figure}[t!]
		\centering	
		\begin{tikzpicture}[shorten >=1pt,node distance=4cm,on grid,auto, thick,scale=0.8, every node/.style={transform shape}]
		\node[block] (q_1)   {\begin{tabular}{c}
		     $s_1$ \\ $x(k+1)=$\\$A_{s_1}x(k)+B_{s_1}w(k)$
		\end{tabular}};
		\node[block] (q_2) [below left = 4cm of q_1] {\begin{tabular}{c}
		     $s_2$ \\ $x(k+1)=$\\$A_{s_2}x(k)+B_{s_2}w(k)$
		\end{tabular}};
		\node[block] (q_3) [below right =4cm of q_1] {\begin{tabular}{c}
		     $s_3$ \\ $x(k+1)=$\\$A_{s_3}x(k)+B_{s_3}w(k)$
		\end{tabular}};

		\path[->]
		(q_1) edge [loop, above, looseness=2] node  {} (q_1)
		(q_1) edge [pos=0.5, bend right, above=0.5,sloped] node {} (q_2)
		(q_1) edge [pos=0.5, bend left, above=0.5,sloped] node {} (q_3)
		
		(q_2) edge [pos=0.5, loop left, looseness=2] node  {} (q_2)
		(q_2) edge [pos=0.5, bend right, above=0.5,sloped] node {} (q_1)
		(q_2) edge [pos=0.5, above=0.5] node {} (q_3)
		
		(q_3) edge [pos=0.5, loop right, looseness=2] node {} (q_3)
		(q_3) edge [pos=0.5, bend left, above=0.5,sloped] node {} (q_1)
		(q_3) edge [pos=0.5, bend left, above=0.5,sloped] node {} (q_2)
		;
		\end{tikzpicture}
		\caption{A system with mode switches governed by an MDP. Each arrow indicates a possible transition. Transition probabilities and actions are omitted.}\label{fig:mdp}
	\end{figure}
We denote $T_\sigma\in\mathbb{R}^{N\times N}$ as the transition probabilities induced by an action $\sigma\in\Sigma$ between state pairs, where $T_\sigma(i,j)=T(s_i,\sigma,s_j)$. If $\sigma$ is not defined on a state $s_i$, $T_\sigma(i,j)=0$ for any $s_j\in S$.

An example of an MDP $\mathcal{M}$ of three states that governs the mode switch of a system of the form (\ref{eq:switched systems}) is shown in Figure \ref{fig:mdp}. It can be observed that there are three system dynamics corresponding to each discrete mode, i.e.
\begin{equation}\label{eq:1}
x(k+1)=A_{s_i}x(k)+B_{s_i}w(k),\text{ for } s_i\in S.
\end{equation}

At each state $s$, there is a set of actions available to choose. Then the nondeterminism of the action selection has to be resolved by a policy $\pi$.

\begin{definition}
	A (randomized) policy $\pi:S\times \Sigma \rightarrow [0,1]$  of an MDP $\mathcal{M}$ is a function that maps every state action pair $(s,\sigma)$ where $s\in S$ and $\sigma\in \Sigma$  with a probability $\pi(s,\sigma)$.    
\end{definition}
By definition, the policy $\pi$ specifies the probability for the next action $\sigma$ to be taken at the current state $s$. As a result, given a policy $\pi$, the MDP $\mathcal{M}$ reduces to a discrete time Markov chain (DTMC) $\mathcal{C}=(S,\hat{s},P)$, where $P$ represents the transition matrix and can be calculated by 
$$
P(s_i,s_j)=\sum_{\sigma \in \Sigma} T(s_i,\sigma,s_j)\pi(s_i,\sigma).
$$
\subsection{Markov Jump Linear Systems}
We define Markov jump linear system \cite{costa2006discrete} as follows.
\begin{definition}
    A Markov jump linear system (MJLS) is a switched system defined in (\ref{eq:switched systems}) with the mode switches governed by a DTMC $\mathcal{C}=(S,\hat{s},P)$. 
\end{definition}

Given a switched system in (\ref{eq:switched systems}) with modes $S$ whose switches are governed by an MDP $\mathcal{M}=(S,\hat{s},\Sigma,T)$ and a policy $\pi$, the resulting system is an MJLS whose mode switches  can be characterized by the DTMC $\mathcal{C}$ induced from the policy $\pi$.


If the system (\ref{eq:switched systems}) is in mode $s_i$, then the probability that it switches to mode $s_j$ is given by $P(s_i,s_j)$.

For MJLS analysis, stability is one of the major concerns. Several notions of stability has been defined in the existing literature \cite{shi2015survey}. In this paper, we are interested in mean square stability as defined below.

\begin{definition}\cite{costa2006discrete}
    An MJLS is said to be mean square stable if
    \begin{equation}\nonumber
    \begin{split}
     &\lim_{k\rightarrow\infty}||E[x(k)-\mu]||_\infty\rightarrow 0 \text{ and }\\
    &\lim_{k\rightarrow\infty}||E[x(k) x'(k)]-C||_\infty\rightarrow 0.       
    \end{split}
    \end{equation}
    for any initial condition $x_0$, where $\mu$ and $C$ are constants.
\end{definition}

\subsection{Semidefinite Program and Bilinear Matrix Inequalities}

In this paper, we use semidefinite programs (SDPs) and bilinear matrix inequalities (BMIs) extensively in our solution approach. We briefly define them in following.

An SDP is an optimization problem with a linear objective, linear equality constraints and a matrix nonnegativity constraint on the variable $y \in \mathbb{R}^{n}$, which can be written as
\begin{align}
    \text{minimize} &\quad c'y\label{eq:sdp objective}\\
  \text{subject to}&  \quad\nonumber\\
    &\quad Ay=b,\label{eq:sdp_equality}\\
   &\quad \displaystyle \sum^n_{i=1} y_i F_i\geq F_0,\label{eq:sdp_positive}
\end{align}
where $F_0,\ldots,F_m \in \mathbb{R}^{p \times p}$, are given symmetric matrices, $A \in \mathbb{R}^{m \times n}$ is a given matrix, and $c \in \mathbb{R}^n, b \in \mathbb{R}^m$ are given vectors. SDPs are convex optimization problems, and can be solved efficiently using interior point methods~\cite{nesterov1994interior,Boyd}. The constraint in~\eqref{eq:sdp_positive} is named as a linear matrix inequality (LMI), and it is a convex constraint in $y$.

A BMI can be written as the following form:
\begin{align*}
 &\quad \displaystyle \sum^n_{i=1} y_j F_i+\sum^m_{j=1} z_j G_j+\sum^n_{i=1}\sum^m_{j=1} y_i z_j H_{ij}\geq F_0,
\end{align*}
where $F_i, G_j, H_{ij} \in \mathbb{R}^{p \times p}$ for $i={1,\ldots,n}$ and $j={1,\ldots,m}$ are given symmetric matrices, and $x \in \mathbb{R}^n, y \in \mathbb{R}^m$ are a vector of variables. A BMI is an LMI in $y$ for fixed $z$ and an LMI in $z$ for fixed $y$. The bilinear terms in a BMI make the feasible set not jointly convex in $y$ and $z$ and it is generally hard to find a feasible solution to a BMI~\cite{vanantwerp2000tutorial}.

\section{Problem Formulation}\label{sec:Problem Formulation}
In traditional MDP literature, finding a policy for an optimized expected cost \cite{puterman2014markov} or to satisfy a specification in temporal logic \cite{baier2008principles} is of the primary concern. However, in this paper, we are concerned with finding a policy $\pi$ in an MDP that governs the switches of a system defined in (\ref{eq:switched systems}). In this case, our objective is to stabilize the switched system defined in~\ref{eq:switched systems}.         


\begin{problem}
Given a switched system as in (\ref{eq:switched systems}) with modes $S$ whose probabilistic transition is described by an MDP $\mathcal{M}=(S,\hat{s},\Sigma,T)$, find a policy $\pi:S\times \Sigma\rightarrow [0,1]$ for $\mathcal{M}$ such that the resulting MJLS with switches defined by the induced DTMC $\mathcal{C}$ is mean square stable.
\end{problem}

\section{Stability Guaranteed Policy Synthesis}\label{sec:Stability Guaranteed Policy Synthesis}
A significant amount of efforts has been devoted to the stability analysis of MJLS in the recent two decades \cite{costa2006discrete,shi2015survey}. We first review some stability conditions that we will leverage to synthesize policies that stabilizes the switched system.
\subsection{Stability Conditions}
Two necessary and sufficient stability conditions for an MJLS are given as the following.
\begin{thm}\cite{costa2006discrete}\label{thm:necceary and sufficient}
Given an MJLS as defined in (\ref{eq:switched systems}) whose mode $s\in S$ makes random transitions described by a DTMC $\mathcal{C}=(S,\hat{s},P)$,  the following assertions are equivalent.


\begin{enumerate}
    \item The MJLS is mean square stable.
    \item $\rho(\mathcal{A})< 1,$ where
    $$
    \mathcal{A} = (P'\otimes I)diag(A_i\otimes A_i),
    $$
    and  $I$ is the identity matrix of a proper dimension.
    \item  There exists a $V=(V_1,\ldots,V_N)\in \mathbb{R}^{n\times n}$ with $V>0$  such that  
	\begin{equation}\label{eq:n&s condition}
	   V-\mathcal{T}(V)>0, 
	\end{equation}
	where 
	$$
	\mathcal{T}_j(V)=\sum_{i=1}^N p_{ij}A_iV_iA'_i \text{ with }  p_{ij}=P(s_i,s_j).
	$$
\end{enumerate}
\end{thm} 

Note that the stability conditions do not depend on either the initial state $\hat{s}$ of the MDP or the initial continuous state $x(0)$.  For computational efficiency, we state a sufficient stability condition as follows.  
\begin{corollary}\cite{costa2006discrete}\label{corollary:sufficient}
Given an MJLS as defined in (\ref{eq:switched systems}) whose mode $s\in S$ makes transitions following a DTMC $\mathcal{C}=(S,\hat{s},P)$,  the MJLS is mean square stable if there exists $\alpha_i>0$ such that the following is satisfied.

	\begin{equation}\label{eq:sufficient condition 1}
\alpha_iI-\sum_{j=1}^N p_{ij}\alpha_jA_iA'_i>0,\; i \in \lbrace{1,\ldots,N\rbrace}.
\end{equation}

\end{corollary}
The condition given in~\eqref{eq:n&s condition} can be checked by solving an SDP with $V_i$ as variables. However, the number of variables for this SDP is $n^2\cdot N$, and finding a feasible solution for the SDP can be time consuming for large $n$ and $N$. On the other hand, the condition in~\eqref{eq:sufficient condition 1} can be checked by solving an SDP with $N$ variables, and the size of the optimization problem is smaller compared to the optimization problem in~\eqref{eq:n&s condition}. Based on these two stability conditions, we propose three different approaches to find the policy in Section \ref{subsec:sos} and Section \ref{subsec:co}.


\subsection{Deterministic policy is not sufficient}
We first show that a deterministic policy, i.e, $\pi:S\rightarrow A$ is not sufficient to guarantee the system stability. It means that there may not exist a deterministic policy to stabilize the system in (\ref{eq:switched systems}), but there exists an randomized policy that achieves stability. 

\begin{thm}
Given a switched system as defined in (\ref{eq:switched systems}) whose mode $s\in S$ makes transitions following a MDP $\mathcal{M}=(S,\hat{s},\Sigma,T)$, deterministic policies are not sufficient to render the system mean square stable.
\end{thm}
\begin{proof}
We prove this theorem by showing an counterexample. Consider a switched system with system dynamics in (\ref{eq:switched systems})

$A_1=
\begin{bmatrix}
    0.99 & -0.56  \\
    -0.19 & 0.73
\end{bmatrix}$,
$A_2=
\begin{bmatrix}
    0.38 & -0.98  \\
    -0.66 & -0.66 
\end{bmatrix}$

The MDP $\mathcal{M}=(S,\hat{s},\Sigma,T)$ where $S=\{s_1,s_2\}$ and $\Sigma=\{\sigma_1,\sigma_2\}$. The transition probabilities induced by action $\sigma_1$ and $\sigma_2$ are

$T_{\sigma_1}=
\begin{bmatrix}
    0.21 & 0.79  \\
    0.90 & 0.10 
\end{bmatrix}$
and
$T_{\sigma_2}=
\begin{bmatrix}
    0.71 & 0.29  \\
    0.13 & 0.87 
\end{bmatrix}$.
 
The deterministic policy that induces a minimal spectral radius  is selecting $\sigma_1$ in both mode $1$ and $2$. The spectral radius $\rho(\mathcal{A})$ of the MJLS induced by this policy is $1.04>1$, which makes the overall system unstable. 

However, the policy that selects $\sigma_1$ in mode 1, and selects $\sigma_1$ in mode 2 with a probability of $0.27$ induces an MJLS that has a spectral radius of $\rho(\mathcal{A})=0.90<1$. So the system is stable according to Theorem \ref{thm:necceary and sufficient}.  Therefore, we conclude that deterministic policies are not sufficient to stabilize the system in~\eqref{eq:switched systems}. 

\end{proof}

\subsection{Policy Synthesis via Bilinear Matrix Inequalities}\label{subsec:sos}
In this section, we formulate a condition based on bilinear matrix inequalities to synthesize stabilizing a policy for the system in~\eqref{eq:switched systems}. The condition is a straightforward generalization of the linear matrix inequalities given in~\eqref{eq:n&s condition}. The following result states that the search for a stabilizing policy can be done by finding a solution to a set of bilinear matrix inequalities.

\begin{thm}\label{thm:bilinear}
Consider a switched system (\ref{eq:switched systems}) whose mode $s\in S$ makes transitions following an MDP $\mathcal{M}=(S,\hat{s},\Sigma,T)$. If there exists matrices $V_i \in \mathbb{R}^{n \times n}$, and $\pi$ such that the following holds:
\begin{align}
  &  V_i >0,\label{eq:positive_sos}\\
   & 	   V-\mathcal{T}(V)>0,\label{eq:positive2_sos}\\
&	\mathcal{T}_j(V)=\sum_{i=1}^N  p_{ij}A_iV_iA_i^{'},\label{eq:positive3_sos}\\
& p_{ij}=\sum_{\sigma \in \Sigma} T(i,\sigma,j)\pi(i,\sigma),\label{eq:induced dtmc_sos}\\
& \sum_{\sigma \in \Sigma} \pi_{i,\sigma}=1,\label{eq:policy_welldefined_sos}\\
&\pi(i,\sigma)\geq 0.\label{eq:policy_nonnegative_sos}
\end{align}
for $i,j= \lbrace{1,\ldots,N\rbrace}$ and $\sigma \in \Sigma$, then the induced MJLS is mean square stable.
\end{thm}

\begin{proof}
Constraints \eqref{eq:induced dtmc_sos}, \eqref{eq:policy_welldefined_sos}, \eqref{eq:policy_nonnegative_sos} construct the induced DTMC $\mathcal{C}$ with transitions governed by $p_{ij}.$ Using the result of Theorem 1, the constraints \eqref{eq:positive_sos}, \eqref{eq:positive2_sos}, \eqref{eq:positive3_sos} ensure that the MJLS is mean square stable for the induced DTMC $\mathcal{C}$. Hence, finding a policy and matrices $V_i$ that satisfies the constraints \eqref{eq:positive_sos}--\eqref{eq:induced dtmc_sos} shows that the  MJLS is mean square stable.
\end{proof}

Note that the constraints given in~\eqref{eq:positive_sos}--\eqref{eq:policy_nonnegative_sos} are BMI constraints due to multiplication between variables $\pi$ and $V$ in ~\eqref{eq:positive2_sos}--\eqref{eq:induced dtmc_sos}, therefore it is hard in general to find a policy by solving the BMI directly. In next section, we propose two scalable approaches based on convex optimization, and discuss their relationship with the BMI in~\eqref{eq:positive_sos}--\eqref{eq:policy_nonnegative_sos}.

\subsection{Policy Synthesis via Convex Optimization}\label{subsec:co}
In this section, we propose two methods to synthesize a policy that stabilizes the switched system in~\eqref{eq:switched systems}. The first method is based on checking feasibility of an SDP, which is an relaxation of the original stability condition. The second method is based on applying a coordinate descent on the variables $V$ and $\pi$. We can use coordinate descent in our case efficiently, as the constraints in~\eqref{eq:positive2_sos}--\eqref{eq:induced dtmc_sos} are LMI constraints if $V$ or $\pi$ is fixed.

\subsubsection{Semidefinite Relaxation}

In the following, we state our semidefinite relaxation to compute a policy that stabilizes the switched system in~\eqref{eq:switched systems}. Our relaxation extends the stability condition given in~\eqref{eq:sufficient condition 1} for an MJLS to a switched system whose mode switches are governed by an MDP. 

\begin{thm}\label{thm:sdp}
Consider a switched system (\ref{eq:switched systems}) whose mode $s\in S$ makes transitions following a MDP $\mathcal{M}=(S,\hat{s},\Sigma,T)$. If there exists $K_{i,\sigma},\alpha_i \in \mathbb{R}>0$ such that
\begin{align}
  &  V_i=\alpha_i I >0,\label{eq:positive_sdp}\\
   & 	   V-\mathcal{T}(V)>0,\label{eq:positive2_sdp}\\
&	\mathcal{T}_j(V)=\sum_{i=1}^N \sum_{\sigma \in \Sigma} T(i,\sigma,j)K_{i,\sigma}A_iA'_i,\label{eq:positive3_sdp}\\
& \sum_{\sigma \in \Sigma} K_{i,\sigma}=\alpha_i,\label{eq:policy_welldefined_sdp}\\
&K_{i,\sigma}\geq 0.\label{eq:policy_nonnegative_sdp}
\end{align}
for $i,j=\lbrace{1,\ldots,N\rbrace}$ and $\sigma \in \Sigma$, then the MJLS is mean square stable.
\end{thm}

\begin{proof}
Suppose that the condition given by constraints \eqref{eq:positive_sos}--\eqref{eq:policy_nonnegative_sos} is satisfied with $V_i=\alpha_i I>0, i=\lbrace{1,\ldots,N \rbrace}$. Then, the constraint~\eqref{eq:positive3_sos} becomes
\begin{align}
&	\mathcal{T}_j(V)=\sum_{i=1}^N \sum_{\sigma \in \Sigma} T(i,\sigma,j)\pi(i,\sigma)\alpha_i A_i A'_i,\label{eq:positive3_bdp}
\end{align}
with variables $\alpha_i>0, i=\lbrace{1,\ldots,N \rbrace}$ and $\pi.$ Note that for a given policy $\pi$ and the induced DTMC $\mathcal{C}$, the constraint in~\eqref{eq:positive3_bdp} is equivalent to the condition given by~\eqref{eq:sufficient condition 1} in Corollary 1. By defining the change of variable $K_{i,\sigma}= \pi(i,\sigma)\cdot\alpha_i $ for $i=\lbrace{1,\ldots,N\rbrace}$ and $\sigma\in\Sigma$, the constraints~\eqref{eq:positive3_sos}--\eqref{eq:policy_nonnegative_sos} are equivalent to the constraints in ~\eqref{eq:positive3_sdp}--\eqref{eq:policy_nonnegative_sdp}. Finding a feasible solution that satisfies the constraints in \eqref{eq:positive_sos}--\eqref{eq:policy_nonnegative_sos} yields a policy $\pi(i,\sigma)=K_{i,\sigma}/\alpha_i$ for $i=\lbrace{1,\ldots,N\rbrace}$ and $\sigma \in \Sigma$, which by construction satisfies the constraints in \eqref{eq:positive_sos}--\eqref{eq:policy_nonnegative_sos}. Therefore, the policy $\pi$ and $V$ ensures that the induced MJLS is mean square stable. 
\end{proof}

The constraints in \eqref{eq:positive_sdp}--\eqref{eq:policy_nonnegative_sdp} are LMIs in the variables $K$ and $\alpha$. Finding a feasible solution of a set of LMIs can be done by solving an SDP. However, this condition is only a sufficient as we restrict the structure of the  matrix $V$, therefore we may not be able to certify the stability of an MJLS even though there may exists a policy that ensures that the MJLS is MSS.





\subsubsection{Coordinate Descent}

In this section we discuss our coordinate descent (CD) approach, and discuss the differences in our method compared to a basic CD algorithm. Recall that a BMI is an LMI if one the variables is fixed, and we can check if the constraints in~\eqref{eq:positive_sos}--\eqref{eq:policy_nonnegative_sos} are feasible for a fixed $V$ or $\pi$. However, applying the basic CD on $V$ and $\pi$ requires the problems are feasible for a fixed $V$ or $\pi$, which is not necessarily true in our case. If the imitial problem is feasible, then we know that $\pi$ stabilizes the MJLS. Therefore, we assume that our initial policy do not stabilize the system.

Our implementation differs from a basic coordinate descent algorithm in the addition of the slack variables to the constraint in~\eqref{eq:positive2_sos}, which ensures that the resulting LMI is feasible for a fixed set of variables, and we use a proximal update between the variables instead of the original update method between $V$ and $\pi$. Details about the proximal update and the convergence guarantees can be found in~\cite{xu2013block}.

We start with an initial guess of the variables $V^0$ and $\pi^0$. Then in each iteration $k=\lbrace{1,\ldots,M\rbrace}$, we solve the following SDP for a fixed $\pi^k$:
\begin{align}
\text{minimize}&\quad -\gamma+\displaystyle\sum^N_{i=1} L||V_i-V^{k-1}_i||_2\\
\text{subject to}&\nonumber\\
  &  \quad V_i >0,\label{eq:positive_cdv}\\
   & \quad	   V-\mathcal{T}(V)\geq \gamma I,\label{eq:positive2_cdv}\\
&	\quad\mathcal{T}_j(V)=\sum_{i=1}^N  p_{ij}A_iV_iA'_i,\label{eq:positive3_cdv}\\
&\quad p_{ij}=\sum_{\sigma \in \Sigma} T(i,\sigma,j)\pi^k(i,\sigma),\label{eq:induced dtmc_cdv}
\end{align}
where $V_i\in \mathbb{R}^{n \times n}, i=\lbrace{1,\ldots,N\rbrace}$ and $\gamma \in \mathbb{R}$ are variables, and $L \in \mathbb{R}$ is a small positive constant. The SDP we solve for a fixed $V^k$ is given as follows:
\begin{align}
\text{minimize}&\quad -\gamma+\displaystyle\sum^N_{i=1}\sum_{\sigma \in \Sigma} L||\pi(i,\sigma)-\pi^{k-1}(i,\sigma)||_2\\
\text{subject to}&\nonumber\\
   & 	 \quad  V^k-\mathcal{T}(V^k)\geq \gamma I,\label{eq:positive2_cdp}\\
&	\quad\mathcal{T}_j(V)=\sum_{i=1}^N  p_{ij}A_iV^k_iA'_i,\label{eq:positive3_cdp}\\
&\quad p_{ij}=\sum_{\sigma \in \Sigma} T(i,\sigma,j)\pi(i,\sigma),\label{eq:induced dtmc_cdp}\\
& \quad\sum_{\sigma \in \Sigma} \pi_{i,\sigma}=1,\label{eq:policy_welldefined_cdp}\\
&\quad\pi(i,\sigma)\geq 0.\label{eq:policy_nonnegative_cdp}
\end{align}
with variables $\pi$ for $i=\lbrace{1,\ldots,N\rbrace}$ and $\sigma \in \Sigma$, and $\gamma \in \mathbb{R}$. After solving each SDP, we update the variables until we converge to a solution or we obtain a solution with $\gamma>0$. If we can find a solution with $\gamma>0$, the conditions~\eqref{eq:positive2_cdv} and~\eqref{eq:positive2_cdp} implies the condition given in~\eqref{eq:positive2_sos}, and the rest of the conditions in~\eqref{eq:positive_sos}--\eqref{eq:policy_nonnegative_sos} are already satisfied in either SDPs that we solve during CD. In this case, we stop the algorithm as the solution given by $V$ and $\pi$ guarantees that the MJLS is MSS. Note that our method is guaranteed to converge as we use the update (1.3b) in~\cite{xu2013block}, however the procedure can converge to a solution with $\gamma\leq 0$, which implies that the CD method cannot certify if the MJLS is MSS.


\section{Examples}\label{sec:Examples}

We demonstrate the proposed approach on three domains:
(1) randomly generated systems, (2) power regulation in wireless networks, and (3) transportation networks. The simulations were performed on a computer with an Intel Core i5-7200u 2.50 GHz processor and 8 GB of RAM with MOSEK~\cite{mosek} as the SDP solver, PENLAB~\cite{fiala2013penlab} as the BMI solver, and using the CVX~\cite{cvx} interface.  In each subsection, we show and compare the results of three proposed methods  by solving the bilinear matrix inequalities (BMI), coordinate descent between $V$ and $\pi$ (CD), and solving the semidefinite relaxation (SDP). 

\subsection{Numerical Examples}

To show the scalability of the proposed method, we generated 10 different systems in \eqref{eq:switched systems} with $n=15$, $N=2$, and $|\Sigma|=2$. The entries of the $A_i$ matrices are randomly selected between $[-0.5,0.5]$, and the transition probabilities for the MDP is generated randomly. 
We show the results of three different methods in Table~\ref{table:random15}. We report number of times that each method was able to find a solution and the average time in seconds for each method when the method is able to find a stabilizing policy.

The results show that the methods with coordinate descent and semidefinite relaxation is faster than the BMI method, and shows that the BMI method does not scale well for systems with large dimensions. The CD and SDP method have similar runtimes, however the CD method is able to find a policy that stabilizes the system in 9 cases out of $10$, and the SDP method can only find a policy in one of the systems. The BMI method had numerical troubles in 6 cases which converged to infeasible solutions.

\begin{table}[t]
\caption{Results for the numerical example with 10 different systems.}
\begin{center}
\scalebox{1.2}{
\begin{tabular}{lll}
\hline
    & Number of successful cases & Average Time (s) \\ \hline
BMP & 4                  & 1087.47    \\ \hline
CDR  & 9                  & 5.35       \\ \hline
SDR & 1                  & 1.11       \\ \hline
\end{tabular}}
\end{center}
\label{table:random15}
\end{table}

\subsection{Transmission Power Regulation in Wireless Networks}
	\begin{figure}
		\centering	
\begin{tikzpicture}[shorten >=1pt,node distance=3cm,on grid,auto, bend angle=20, thick,scale=0.7, every node/.style={transform shape}] 
				\node[state] (s0)   {$t_1$}; 
				\node[state] (s1) [below right= 2 and 1 of s0] {$r_1$}; 
				\node[state] (s2) [above right  = 2 of s0]  {$t_2$}; 
				\node[state] (s3) [ right= of s1] {$r_2$}; 
				\path[->]
				
				(s0) edge node {$g_{11}$} (s1) 
				(s0) edge node [below] {$g_{12}$} (s3) 
				(s2) edge [dashed] node {$g_{21}$} (s3) 
                (s2) edge [dashed] node {$g_{22}$} (s1) 
				; 

				\end{tikzpicture} 
		\caption{A wireless network with two transmitters and two receivers.}\label{fig:wireless network}
	\end{figure}
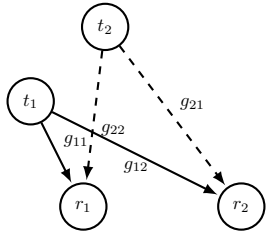
	
Let's consider a wireless network with $n$ nodes as shown in Figure \ref{fig:wireless network}. If the transmitter $t_i$ transmits with power $x_i$ to its corresponding receiving node $r_i$, its quality-of-service (QoS) can be characterized by the Signal-to-Interference-and-Noise-Ratio (SINR) as the following.
\begin{equation}\label{eq:SINR}
\Gamma_i=\frac{g_{ii}x_i}{\sum_{j\in\mathcal{N}_i}g_{ji}x_j+v_i},
\end{equation}
where $g_{ij}\in(0,1]$ denotes the path loss on the communication link between transmitter $t_i$ and receiver $r_j$ due to distance, shadowing and fading, $v_i$ denotes the thermal noise at the receiver $r_i$ and $\mathcal{N}_i$ denotes the set of transmitters different from $t_i$ that interfere with the receiver $r_i$. 

To achieve reliable communication, it is desired that SINR at the receiver $r_i$ is no less than a threshold $\gamma_i$, i.e.,
\begin{equation}\label{eq:QoS}
\frac{g_{ii}x_i}{\sum_{j\in\mathcal{N}_i}g_{ij}x_j+v_i}\geq\gamma_i. 
\end{equation}
If the path losses are constants, a well-known distributed power allocation algorithm, which is called Foschini-Miljanic (FM) algorithm was proposed in \cite{foschini1993simple} as shown below.

\begin{equation}\label{eq:FM algorithm}
    x_i(k+1)=(1-\lambda_i)x_i(k)+\lambda_i\gamma_i(\sum_{j\in\mathcal{N}_i}\frac{g_{ji}}{g_{ii}}x_j(k)+\frac{v_i}{g_{ii}}),
\end{equation}
where $\lambda_i\in(0,1]$. In matrix form, we write (\ref{eq:FM algorithm}) into

\begin{equation}\label{eq:FM algorithm matrix form}
\begin{split}
    x(k+1)&=(I-\Lambda H)x(k)+\Lambda\eta v(k)\\
          &=Ax(k)+Bv(k).
\end{split}
\end{equation}
where $I$ is the identity matrix, $\Lambda=diag(\lambda_i)$ and $\eta=diag(\frac{\gamma_i}{g_{ii}})$. $H$ is an $n\times n$ matrix defined by
\[
H_{ij} = 
\begin{cases}
1 ,& \text{if } i=j,\\
-\gamma_i\frac{g_{ji}}{g_{ii}}             & \text{otherwise.}
\end{cases}
\]
FM algorithm can find the smallest power vector $x$ in the element-wise sense to satisfy QoS requirement (\ref{eq:QoS}) when (\ref{eq:FM algorithm matrix form}) is stable \cite{foschini1993simple}. However, in practice, path losses $g_{ij}$ are uncertain and can fluctuate randomly due to environmental uncertainties or different antenna configurations \cite{goldsmith2005wireless}. With the recent advances in mm-wave communications, mechanically or electrically-steerable adaptive antennas are being applied in practice \cite{andrews2014will}. As a result, a more realistic model can be characterized as 
\begin{equation}
    x(k+1)=A_{s_k}x(k)+B_{s_k}v(k),
\end{equation}
where the path loss matrix $g\in\mathbb{R}^{n\times n}$ with $g_{ij}$ as defined before jumps randomly among $N$ different values based on selected antenna configurations. The switch between mode $s_i$ to $s_j$ given an antenna configuration $\sigma\in \Sigma$ is $T(s_i,\sigma,s_j)$ which is governed by an MDP $\mathcal{M}=(S,\hat{s},\Sigma,T)$ with  transition probabilities induced by action $\sigma_1$ and $\sigma_2$ are

$T_{\sigma_1}=
\begin{bmatrix}
    0.9 & 0.1  \\
    0.1 & 0.9 
\end{bmatrix}$
and
$T_{\sigma_2}=
\begin{bmatrix}
    0.3 & 0.7  \\
    0.6 & 0.4 
\end{bmatrix}$.

Therefore, the objective in this wireless network is to regulate the transmission power and guarantee the stability of the power vector by finding a policy to switch among different antenna configurations of the transmitter nodes.

The problem we consider has four nodes, which corresponds to four continuous states, and two modes. We fixed the MDP model and  repeated the example with 50 different continuous dynamics of the switched system. We report number of times that each method was able to find a solution and the average time in seconds for each method in Table~\ref{table:wireless}. We set timeout to 300 seconds. As indicated in Table~\ref{table:wireless}, the SDP method could not find a policy in most cases, as it is only able to certify the stability of the system in four cases. On the other hand, the BMI and CD method can find a solution in most of the cases, and the CD approach is faster than BMI method in average.

\begin{table}[t]
\caption{Results for the wireless network example with 50 different systems.}
\begin{center}
\scalebox{1.2}{
\begin{tabular}{lll}
\hline
    & Number of successful cases & Average Time (s) \\ \hline
BMI & 48                 & 6.23      \\ \hline
CD  & 47                 & 1.87        \\ \hline
SDP & 4                  & 0.19        \\ \hline
\end{tabular}}
\end{center}
\label{table:wireless}
\end{table}

\subsection{Linear Transportation Network}
We adapt this example from \cite{rantzer2012optimizing}. Consider a transportation network connecting four buffers as shown in Figure \ref{fig:transportation network}. 
	\begin{figure}
		\centering	
\begin{tikzpicture}[shorten >=1pt,node distance=3cm,on grid,auto, bend angle=20, thick,scale=0.7, every node/.style={transform shape}] 
				\node[state] (s0)   {$x_1$}; 
				\node[state] (s1) [right= of s0] {$x_2$}; 
				\node[state] (s2) [below = of s1]  {$x_3$}; 
				\node[state] (s3) [left= of s2] {$x_4$}; 
				\path[->]
				
				(s0) edge node {} (s2) 
				
				(s1) edge node {} (s0) 
				
				(s1) edge node  {} (s2) 
				(s2) edge  node {} (s1) 
                (s2) edge  node {} (s3) 
                (s3) edge  node {} (s2) 
				; 

				\end{tikzpicture} 
		\caption{A transportation network. Each node represents a buffer. Each arrow indicates a transportation link.}\label{fig:transportation network}
	\end{figure}
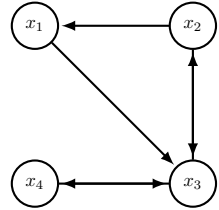
The continuous-time dynamics for this transportation network are described by $\dot{x}=Ax$, where 
\begin{align}\label{eq:transportation networks}
A =\begin{bmatrix}
  -1-l_{31} & l_{12} & 0 & 0  \\
  0 & 2-l_{12}-l_{32} & l_{23} & 0  \\
  l_{31} & l_{32} & 3-l_{23}-l_{43} & l_{34}  \\
  0 & 0 & l_{43} & -4-l_{34}
\end{bmatrix}.
\end{align}
The states $x$ represents the quantity of the contents in the buffers and $l_{ij}$ represents the rate of transfer from buffer $j$ to $i$. 

We consider the discrete time version of the model in (\ref{eq:transportation networks}) which can be obtained in a standard way \cite{antsaklis2006linear}. The sampling time is $0.1$. There are two actions that may affect the rate of transfer probabilistically which result in two different matrices $A$. This can be modeled as a switched linear system with transitions governed by an MDP $\mathcal{M}$. The objective is to guarantee the stability for each buffer. 

We discuss the results in one particular case of a network. The system we consider has two discrete modes. In the first mode, the rate of transfer for $l_{12}, l_{23}$ and $l_{31}$ are set to zero. In the second mode, the rate of transfer for $l_{32}, l_{34}$ and $l_{43}$ are set to zero. The transition probabilities induced by action $\sigma_1$ and $\sigma_2$ are

$T_{\sigma_1}=
\begin{bmatrix}
    0.5 & 0.5  \\
    0.6 & 0.4 
\end{bmatrix}$
and
$T_{\sigma_2}=
\begin{bmatrix}
    0.8 & 0.2  \\
    0.2 & 0.8 
\end{bmatrix}$.

After 5 iterations, the CD method is able to find a feasible policy that selects $\sigma_1$ in mode 1 with a probability of $0.22$, and selects $\sigma_1$ in mode with a probability of $0.13$. The spectral radius of the MJLS induced by the policy is $0.997<1$, which ensures that the MJLS is stable. The solution time for the CD method is 1.9 seconds. The BMI method converges to a solution that assigns a negative probability to $\sigma_1$ in mode 1, and therefore it is infeasible. The solution time for the BMI approach is 22.07 seconds. The SD method was infeasible, even though there exists a policy that stabilizes the MJLS.

 \begin{table}[t]
\caption{Results for the linear transportation network example with 50 different systems.}
\begin{center}
\scalebox{1.2}{
\begin{tabular}{lll}
\hline
    & Number of successful cases & Average Time (s) \\ \hline
BMI & 28                 & 19.15      \\ \hline
CD  & 41                & 6.78        \\ \hline
SDP & 3                 & 0.32        \\ \hline
\end{tabular}}
\end{center}
\label{table:transportation}
\end{table}

We show the results of three different methods in Table~\ref{table:transportation} similar to the previous examples. In this example, the MDP has 4 discrete states and two actions. We fixed the MDP model and repeated the example with 50 different continuous dynamics. We report number of times that each method was able to find a solution and the average time in seconds for each method in Table~\ref{table:transportation}. Similar to the previous example, we set timeout to 300 seconds. We note that all methods are able to find a stabilizing policy in fewer cases compared to the wireless network example. We also observe that the average time was larger compared to the previous example, even though both examples have 4 continuous states. Like previous examples, the SDP approximation does not find a stabilizing policy in almost all cases. The CD method was able to find a solution in more cases compared to solving the BMI directly, and it is faster when both BMI and CD methods were able to find a policy that stabilizes the system.

\section{Conclusion}\label{sec:Conclusion}
In this paper, we consider a switched linear system whose mode switches are governed by an Markov decision process (MDP). The objective is to find a policy in the MDP such that the switched system is guaranteed to be stable. An MDP reduces to a discrete time Markov chain with a given policy and the switched system becomes an Markov jump linear system (MJLS). So we leverage the stability conditions in MJLS and propose three different approaches to compute the stabilizing policy. Our numerical experiments show that solving for bilinear inequalities is not a practical approach for systems with large dimensions. They also show that the relaxations that has been proposed in the literature cannot certify the stability of the system in most of the cases. 

For future works, we will continue to investigate how to incorporate additional temporal logic constraint on mode switches. We will also study how to find a stabilizing policy that optimizes the cost of mode switches.   
\bibliographystyle{IEEEtran}
\bibliography{ref}

\begin{thebibliography}{10}
\providecommand{\url}[1]{#1}
\csname url@samestyle\endcsname
\providecommand{\newblock}{\relax}
\providecommand{\bibinfo}[2]{#2}
\providecommand{\BIBentrySTDinterwordspacing}{\spaceskip=0pt\relax}
\providecommand{\BIBentryALTinterwordstretchfactor}{4}
\providecommand{\BIBentryALTinterwordspacing}{\spaceskip=\fontdimen2\font plus
\BIBentryALTinterwordstretchfactor\fontdimen3\font minus
  \fontdimen4\font\relax}
\providecommand{\BIBforeignlanguage}[2]{{%
\expandafter\ifx\csname l@#1\endcsname\relax
\typeout{** WARNING: IEEEtran.bst: No hyphenation pattern has been}%
\typeout{** loaded for the language `#1'. Using the pattern for}%
\typeout{** the default language instead.}%
\else
\language=\csname l@#1\endcsname
\fi
#2}}
\providecommand{\BIBdecl}{\relax}
\BIBdecl

\bibitem{liberzon2003switching}
D.~Liberzon, \emph{Switching in systems and control}.\hskip 1em plus 0.5em
  minus 0.4em\relax Springer Science \& Business Media, 2003.

\bibitem{sun2006switched}
Z.~Sun, \emph{Switched linear systems: control and design}.\hskip 1em plus
  0.5em minus 0.4em\relax Springer Science \& Business Media, 2006.

\bibitem{sun2005analysis}
Z.~Sun and S.~S. Ge, ``Analysis and synthesis of switched linear control
  systems,'' \emph{Automatica}, vol.~41, no.~2, pp. 181--195, 2005.

\bibitem{zegers2018distributed}
F.~M. Zegers, H.-Y. Chen, P.~Deptula, and W.~E. Dixon, ``Distributed
  coordination of a multi-agent system with intermittent communication: A
  switched systems approach,'' in \emph{ASME 2018 Dynamic Systems and Control
  Conference}.\hskip 1em plus 0.5em minus 0.4em\relax American Society of
  Mechanical Engineers, 2018.

\bibitem{zhang2017energy}
D.~Zhang, P.~Shi, W.-A. Zhang, and L.~Yu, ``Energy-efficient distributed
  filtering in sensor networks: A unified switched system approach,''
  \emph{IEEE Transactions on Cybernetics}, vol.~47, no.~7, pp. 1618--1629,
  2017.

\bibitem{cetinkaya2018analysis}
A.~Cetinkaya, H.~Ishii, and T.~Hayakawa, ``Analysis of stochastic switched
  systems with application to networked control under jamming attacks,''
  \emph{IEEE Transactions on Automatic Control}, 2018.

\bibitem{wu2018privacy}
B.~Wu and H.~Lin, ``Privacy verification and enforcement via belief
  abstraction,'' \emph{IEEE control systems letters}, vol.~2, no.~4, pp.
  815--820, 2018.

\bibitem{zhang2008l_}
L.~Zhang and P.~Shi, ``$ l_2-l_\infty$ model reduction for switched lpv systems
  with average dwell time,'' \emph{IEEE Transactions on Automatic Control},
  vol.~53, no.~10, pp. 2443--2448, 2008.

\bibitem{puterman2014markov}
M.~L. Puterman, \emph{Markov decision processes: discrete stochastic dynamic
  programming}.\hskip 1em plus 0.5em minus 0.4em\relax John Wiley \& Sons,
  2014.

\bibitem{olfati2004consensus}
R.~Olfati-Saber and R.~M. Murray, ``Consensus problems in networks of agents
  with switching topology and time-delays,'' \emph{IEEE Transactions on
  automatic control}, vol.~49, no.~9, pp. 1520--1533, 2004.

\bibitem{costa2006discrete}
O.~L.~V. Costa, M.~D. Fragoso, and R.~P. Marques, \emph{Discrete-time Markov
  jump linear systems}.\hskip 1em plus 0.5em minus 0.4em\relax Springer Science
  \& Business Media, 2006.

\bibitem{bolzern2015positive}
P.~Bolzern, P.~Colaneri \emph{et~al.}, ``Positive markov jump linear systems,''
  \emph{Foundations and Trends{\textregistered} in Systems and Control},
  vol.~2, no. 3-4, pp. 275--427, 2015.

\bibitem{shi2015survey}
P.~Shi and F.~Li, ``A survey on markovian jump systems: modeling and design,''
  \emph{International Journal of Control, Automation and Systems}, vol.~13,
  no.~1, pp. 1--16, 2015.

\bibitem{saravanakumar2017stability}
R.~Saravanakumar, M.~S. Ali, C.~K. Ahn, H.~R. Karimi, and P.~Shi, ``Stability
  of markovian jump generalized neural networks with interval time-varying
  delays,'' \emph{IEEE transactions on neural networks and learning systems},
  vol.~28, no.~8, pp. 1840--1850, 2017.

\bibitem{vanantwerp2000tutorial}
J.~G. VanAntwerp and R.~D. Braatz, ``A tutorial on linear and bilinear matrix
  inequalities,'' \emph{Journal of process control}, vol.~10, no.~4, pp.
  363--385, 2000.

\bibitem{nesterov1994interior}
Y.~Nesterov and A.~Nemirovskii, \emph{Interior-point polynomial algorithms in
  convex programming}.\hskip 1em plus 0.5em minus 0.4em\relax Siam, 1994,
  vol.~13.

\bibitem{razaviyayn2013unified}
M.~Razaviyayn, M.~Hong, and Z.-Q. Luo, ``A unified convergence analysis of
  block successive minimization methods for nonsmooth optimization,''
  \emph{SIAM Journal on Optimization}, vol.~23, no.~2, pp. 1126--1153, 2013.

\bibitem{shen2017disciplined}
X.~Shen, S.~Diamond, M.~Udell, Y.~Gu, and S.~Boyd, ``Disciplined multi-convex
  programming,'' in \emph{2017 29th Chinese Control And Decision Conference
  (CCDC)}.\hskip 1em plus 0.5em minus 0.4em\relax IEEE, 2017, pp. 895--900.

\bibitem{lin2009stability}
H.~Lin and P.~J. Antsaklis, ``Stability and stabilizability of switched linear
  systems: a survey of recent results,'' \emph{IEEE Transactions on Automatic
  control}, vol.~54, no.~2, pp. 308--322, 2009.

\bibitem{Boyd}
S.~Boyd and L.~Vandenberghe, \emph{Convex Optimization}.\hskip 1em plus 0.5em
  minus 0.4em\relax Cambridge University Press, 2004.

\bibitem{baier2008principles}
C.~Baier and J.-P. Katoen, \emph{Principles of model checking}, 2008.

\bibitem{xu2013block}
Y.~Xu and W.~Yin, ``A block coordinate descent method for regularized
  multiconvex optimization with applications to nonnegative tensor
  factorization and completion,'' \emph{SIAM Journal on imaging sciences},
  vol.~6, no.~3, pp. 1758--1789, 2013.

\bibitem{mosek}
\BIBentryALTinterwordspacing
M.~ApS, \emph{The MOSEK optimization toolbox for MATLAB manual. Version 8.1.},
  2017. [Online]. Available: \url{http://docs.mosek.com/8.1/toolbox/index.html}
\BIBentrySTDinterwordspacing

\bibitem{fiala2013penlab}
J.~Fiala, M.~Ko{\v{c}}vara, and M.~Stingl, ``Penlab: A matlab solver for
  nonlinear semidefinite optimization,'' \emph{arXiv preprint arXiv:1311.5240},
  2013.

\bibitem{cvx}
I.~CVX~Research, ``{CVX}: Matlab software for disciplined convex programming,
  version 2.0,'' \url{http://cvxr.com/cvx}, Aug. 2012.

\bibitem{foschini1993simple}
G.~J. Foschini and Z.~Miljanic, ``A simple distributed autonomous power control
  algorithm and its convergence,'' \emph{IEEE transactions on vehicular
  Technology}, vol.~42, no.~4, pp. 641--646, 1993.

\bibitem{goldsmith2005wireless}
A.~Goldsmith, \emph{Wireless communications}.\hskip 1em plus 0.5em minus
  0.4em\relax Cambridge university press, 2005.

\bibitem{andrews2014will}
J.~G. Andrews, S.~Buzzi, W.~Choi, S.~V. Hanly, A.~Lozano, A.~C. Soong, and
  J.~C. Zhang, ``What will 5g be?'' \emph{IEEE Journal on selected areas in
  communications}, vol.~32, no.~6, pp. 1065--1082, 2014.

\bibitem{rantzer2012optimizing}
A.~Rantzer, ``Optimizing positively dominated systems,'' in \emph{2012 IEEE
  51st IEEE Conference on Decision and Control (CDC)}.\hskip 1em plus 0.5em
  minus 0.4em\relax IEEE, 2012, pp. 272--277.

\bibitem{antsaklis2006linear}
P.~J. Antsaklis and A.~N. Michel, \emph{Linear systems}.\hskip 1em plus 0.5em
  minus 0.4em\relax Springer Science \& Business Media, 2006.

\end{thebibliography}
\end{document}